\documentclass[11pt,a4paper]{article}

\usepackage{psfrag}
\usepackage{amsmath,amssymb,cite}
\usepackage{latexsym}
\usepackage{graphicx}
\usepackage{lscape}
\usepackage{caption}
\usepackage{subcaption} 

\usepackage{hyperref}
\hypersetup{
    colorlinks=true,
    linkcolor=blue,
    filecolor=magenta,      
    urlcolor=cyan,
    citecolor=red,
}
\DeclareMathOperator*{\argmin}{argmin}

\newcommand{\st}{\widetilde{s}}
\newcommand{\Ht}{\widetilde{H}}
\newcommand{\lambdat}{\widetilde{\lambda}}
\newcommand{\varphit}{\widetilde{\varphi}}
\newcommand{\sing}{\mbox{\footnotesize sing}}

\newcommand{\zb}{\overline{z}}

\newcommand{\ds}{\displaystyle}
\newcommand{\nexto}{\kern -0.54em}
\newcommand{\dR}{{\rm {I\ \nexto R}}}

\newcommand{\dZ}{{\cal Z \kern -0.7em Z}}
\newcommand{\dC}{{\rm\hbox{C \kern-0.8em\raise0.2ex\hbox{\vrule
height5.4pt width0.7pt}}}}
\newcommand{\dQ}{{\rm\hbox{Q \kern-0.85em\raise0.25ex\hbox{\vrule
height5.4pt width0.7pt}}}}
\newcommand{\proofbox}{\hspace{\fill}{$\Box$}}

\newtheorem{corollary}{Corollary}
\newtheorem{proposition}{Proposition}

\newtheorem{fact}{Fact}
\newtheorem{remark}{Remark}

\newenvironment{proof}{Proof.}{\proofbox}

\begin{document}

\author{Authors}

\author{
C. Yal{\c c}{\i}n Kaya\footnote{Mathematics, UniSA STEM, University of South Australia, Australia. E-mail: yalcin.kaya@unisa.edu.au\,.}
}

\title{\vspace{-10mm}\bf Observer Path Planning for Maximum Information}

\maketitle

\centerline{\large Dedicated to the memory of Alex Rubinov}

\begin{abstract} 
\noindent {\sf This paper is concerned with finding an optimal path for an observer, or sensor, moving at a constant speed, which is to estimate the position of a stationary target, using only bearing angle measurements. The generated path is optimal in the sense that, along the path, information, and thus the efficiency of a potential estimator employed, is maximized. In other words, an observer path is deemed optimal if it maximizes information so that the location of the target is estimated with smallest uncertainty, in some sense. We formulate this problem as an optimal control problem maximizing the determinant of the Fisher information matrix, which is one of the possible measures of information.  We derive analytical results for optimality using the Maximum Principle.  We carry out numerical experiments and discuss the multiple (locally) optimal solutions obtained.  We verify graphically that the necessary conditions of optimality are verified by the numerical solutions.  Finally we provide a comprehensive list of possible extensions for future work.}
\end{abstract}

\begin{verse} 
\end{verse}
\begin{verse}
\end{verse}

\pagestyle{myheadings}
\markboth{}{\sf\scriptsize Observer Path Planning for Maximum Information\ \ by C. Y. Kaya}

\section{Introduction and Problem Statement}
\label{problem}

We aim to find the optimal path of an observer, or sensor, moving at a constant speed, which is to estimate the position of a {\em stationary} target, using only bearing angle measurements.  The process of this estimation is referred to as {\em target localization}, since the target is stationary.  The generated path is optimal in the sense that, along the path, the observer, or sensor, is efficient. The efficiency is achieved when some measure of information is maximized, as will be discussed in detail later.  The notation we adopt below is a broad combination of the notations used in, which are now classical, references~\cite{HamLiuHilGon1989, OshDav1999, PasCap1998}.

The {\em observer position} in the two-dimensional space (or the plane) at time $t$ is given as $(x_s(t),y_s(t))$, as depicted in~Figure~\ref{fig:bearing}.  The fixed target coordinates are denoted by $(x_T,y_T)$.  The objective of the observer is to estimate the target coordinates by using a sequence of $N$ {\em bearing angle} measurements $\{\beta_1,\beta_2,\ldots,\beta_N\}$, taken over the time interval $[0,t_f]$, where $t_f$ is referred to as the {\em final} or {\em terminal time}, which is fixed.

\begin{figure}[t]
\begin{center}
\psfrag{theta}{$\theta(t)$}
\psfrag{beta}{$\beta(t)$}
\psfrag{xsys}{$(x_s(t),y_s(t))$}
\psfrag{xTyT}{$(x_T,y_T)$}
\psfrag{v}{$v$}
\psfrag{Target}{Target}
\psfrag{Observer}{Observer}
\includegraphics[width=70mm]{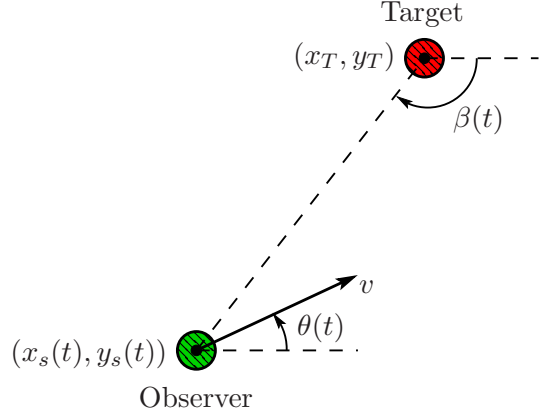}
\end{center}
\caption{\sf Geometric configuration of the observer and target.}
\label{fig:bearing}
\end{figure}

Suppose that the observer, for example a ship, an autonomous underwater vehicle, or an uninhabited aerial vehicle (a UAV or drone), is modelled as a point mass which moves at a constant speed $v$.  Then the equations of motion of the observer are given by
\begin{eqnarray*}
&& \dot{x}_s(t) = v\,\cos\theta(t)\,,\ \ x_s(0) = x_{s_0}\,, \\[1mm]
&& \dot{y}_s(t) = v\,\sin\theta(t)\,,\ \ \,\,y_s(0) = y_{s_0}\,,
\end{eqnarray*}
where $\dot{x}_s = dx_s/dt$, etc., and the {\em observer course} $\theta(t)$ is the angle the observer's velocity vector makes with the $x$-axis in the counter-clockwise direction at time $t$.

Next, we define the {\em relative observer position}\,:
\[
(x(t),y(t)) := (x_s(t) - x_T,\  y_s(t) - y_T)\,.
\]
Then the observer's motion equations can simply be written in terms of the relative coordinates as
\begin{eqnarray}
&& \dot{x}(t) = v\,\cos\theta(t)\,,\ \ x(0) = x_0\,, \label{eqnx} \\[1mm]
&& \dot{y}(t) = v\,\sin\theta(t)\,,\ \ \,y(0) = y_0\,,\label{eqny}
\end{eqnarray}
where $x_0 =  x_s(0) - x_T$ and $y_0 =  y_s(0) - y_T$.  From the geometry in~Figure~\ref{fig:bearing}, the measurement equation for the, now time-dependent, {\em bearing angle} $\beta(t)$ is
\begin{equation}  \label{beta}
  \tan(\beta(t) + w(t)) = \frac{y(t)}{x(t)}\,,
\end{equation}
where, for any $t\in[0,t_f]$, $w(t)$ comes from a zero-mean Gaussian distribution with fixed variance $\sigma^2$.  So, it is assumed that the bearing angle measurements are taken continuously throughout the time horizon $[0,t_f]$; i.e., $N\to\infty$.  We note that in practice Equation~\eqref{beta} is solved numerically by using the two-argument variant of the $\arctan$ function, which takes into account the signs of $x(t)$ and $y(t)$.

An observer path is deemed optimal if it maximizes information so that the location of the target is estimated with smallest uncertainty in some sense.  One possible measure of information is the determinant of the {\em Fisher information matrix}, the matrix which is also considered in \cite{HamLiuHilGon1989, PasCap1998}:
\begin{equation}  \label{FIM}
F(\theta) = \frac{1}{\sigma^2}\left[
\begin{array}{rr}
\ds\int_0^{t_f}\frac{y^2(\tau)}{(x^2(\tau) + y^2(\tau))^2}\,d\tau & \quad  
\ds -\int_0^{t_f}\frac{x(\tau)\,y(\tau)}{(x^2(\tau) + y^2(\tau))^2}\,d\tau \\[6mm]
\ds -\int_0^{t_f}\frac{x(\tau)\,y(\tau)}{(x^2(\tau) + y^2(\tau))^2}\,d\tau  & \quad
\ds\int_0^{t_f}\frac{x^2(\tau)}{(x^2(\tau) + y^2(\tau))^2}\,d\tau
\end{array}
\right],
\end{equation}
subject to \eqref{eqnx}--\eqref{eqny}.  It should be noted that $F(\theta)$ is the inverse of the covariance matrix for the maximum likelihood estimation error, and thus maximizing the determinant of $F(\theta)$ is equivalent to minimizing the area of the {\em uncertainty ellipse} defined by the (positive-definite) covariance matrix.  Therefore the maximization of the determinant of $F(\theta)$ effectively minimizes the error made in estimating the position of a target (by a potential estimator), in the sense explained here. 

In \cite{HamLiuHilGon1989, OshDav1999}, direct discretizations of $F(\theta)$ and the ODEs \eqref{eqnx}--\eqref{eqny} are carried out; namely the Riemannian sum for the integrals in \eqref{FIM} and the corresponding Euler discretization of \eqref{eqnx}--\eqref{eqny} are considered.  Subsequently, approximate solutions maximizing the determinant of $F(\theta)$ are obtained by employing numerical techniques, particularly decent methods, for finite-dimensional optimization.

In~\cite{HamLiuHilGon1989}, an expression for the optimal path is derived via necessary optimality conditions for the unconstrained optimal control problem of minimizing a simpler minorant function, i.e., a function which bounds pointwise from below the determinant of~\eqref{FIM}.  In~\cite{PasCap1998}, the optimality conditions for the problem of minimizing the determinant of~\eqref{FIM} are written down for the case when both the state and control variables are unconstrained, and Euler discretization of the (Euler--Lagrange) optimality conditions are solved, again by using descent methods.

In the current paper we develop a new optimal control model with the extensions of (i)~prescription of the initial course (or heading angle) and (ii) imposition of an upper bound for the curvature, i.e., a lower bound of the turning radius.  By using the Maximum Principle, we obtain the necessary conditions of optimality for the new problem, presenting results about the optimal control.  Obviously the new problem is tractable only numerically; so, via Euler discretization, we obtain approximate solutions over various time horizons, by solving large-scale optimization problems.  We present numerical results about the type of optimal control structure: We show that the optimal control for the particular example we are studying can be a single {\em bang} arc, a {\em bang--singular} arc or a {\em bang--singular--bang} arc.

The new model and its solution techniques introduced in the present paper can be extended to cater for state constraints, multiple observers (see, e.g. \cite{Dogancay2012}) and moving targets---see Section~\ref{conclusion} for a discussion of future directions.  Optimal paths generated for maximum information could be instrumental in testing the performance of estimators~\cite{RisAru2003}, both for the cases of localization (stationary target) and tracking (moving target).  So, the significance of the study presented is attested by many real-life applications.

The paper is organized as follows.  In Section~\ref{formulation}, we formulate the problem we have stated above as an optimal control problem for both the cases when (i) the initial course is unspecified and the control is unconstrained and (ii) the initial course is specified and the control is constrained by a bound.  We also derive the necessary conditions of optimality via the Maximum Principle for each of the two problems, and present some analytical results.  In Section~\ref{formulation}, we solve the two problems numerically for an example configuration from the literature and verify the optimality conditions for each solution that we obtain.  In Section~\ref{conclusion}, after summarising what has been achieved in the current paper, we provide a comprehensive list of possible extensions for future work.

\newpage
\section{Formulation and Analysis as an Optimal Control Problem}
\label{formulation}

The problem of maximizing the determinant of the Fisher information matrix $F(\theta)$ subject to \eqref{eqnx}--\eqref{eqny} can be reformulated as an optimal control problem in standard form as follows.  First, define new variables:
\begin{eqnarray*}
&& z_1(t) := \int_0^{t}\frac{x^2(\tau)}{(x^2(\tau) + y^2(\tau))^2}\,d\tau\,,\quad z_1(0) = 0\,, \\[2mm]
&& z_2(t) := \int_0^{t}\frac{y^2(\tau)}{(x^2(\tau) + y^2(\tau))^2}\,d\tau\,, \quad z_2(0) = 0\,, \\[2mm]
&& z_3(t) := \int_0^{t}\frac{x(\tau)\,y(\tau)}{(x^2(\tau) + y^2(\tau))^2}\,d\tau\,, \quad z_3(0) = 0\,, 
\end{eqnarray*}
for $t\in[0,t_f]$.  With these new variables, the determinant of $F(\theta)$ can simply be written as
\[
\det(F(\theta)) = \frac{1}{\sigma^4}\left(z_1(t_f)\,z_2(t_f) - z_3^2(t_f)\right).
\]
Then the problem of finding an optimal observer path for target localization can now be expressed as the maximization of the functional $\det(F(\theta))$ as follows.
\[
\mbox{(P1) }\left\{\begin{array}{rll}
\ds\max & \ \sigma^4\det(F(\theta))\  & \hspace*{-15mm}\equiv\ \ds\min\ \left(z_3^2(t_f) - z_1(t_f)\,z_2(t_f)\right)    \\[4mm] 
\mbox{subject to} & 
\ \dot{x}(t) = v\,\cos\theta(t)\,, &\ \ x(0) = x_0\,, \\[2mm]
& \ \dot{y}(t) = v\,\sin\theta(t)\,, &\ \ y(0) = y_0\,, \\[2mm]
& \ \ds\dot{z}_1(t) = \frac{x^2(t)}{(x^2(t) + y^2(t))^2}\,, &\ \ z_1(0) = 0\,, \\[3mm]
& \ \ds\dot{z}_2(t) = \frac{y^2(t)}{(x^2(t) + y^2(t))^2}\,, &\ \ z_2(0) = 0\,, \\[3mm]
& \ \ds\dot{z}_3(t) = \frac{x(t)\,y(t)}{(x^2(t) + y^2(t))^2}\,, &\ \ z_3(0) = 0\,, \mbox{ a.e. } t\in[0,t_f]\,.
\end{array} \right.
\]
This is an optimal control problem in standard form, where the {\em state variables} are $x$, $y$, $z_1$, $z_2$ and $z_3$, and the {\em control variable} is $\theta$.  The objective functional is given in terms of the terminal values of the state variables; therefore, it is in the so-called {\em Mayer form}.  Because of the complexity of Problem~(P1), it is possible to find a solution only numerically.  Experiments show that finding an accurate numerical solution to Problem~(P1), even with mild extensions, is a challenging task.  

In \cite{HamLiuHilGon1989, PasCap1998, OshDav1999}, the problem the authors solve amounts to a {\em coarse} Euler discretization of Problem~(P1).  Neither in~\cite{HamLiuHilGon1989, PasCap1998, OshDav1999} nor, to the knowledge of the author, in the rest of the literature, the initial course $\theta(0)$ is specified.  However, as will be motivated by the numerical solution of Problem~(P1) presented in Section~\ref{sec:P1}, absence of the initial heading angle (or course) of the vehicle is not quite realistic.  So we pose below a second problem in which this modelling oversight is addressed.
\[
\mbox{(P2) }\left\{\begin{array}{rll}
\mbox{Solve} & \ \mbox{Problem~(P1)} & \hspace*{-20mm}\ \  \\[4mm] 
\mbox{subject to} & 
\ \dot{\theta}(t) = u(t)/v\,, & \theta(0) = \theta_0\,, \mbox{ a.e. } t\in[0,t_f]\,.
\end{array} \right.
\]
In Problem~(P2), $\theta$ is no longer the control variable but a state variable.  The {\em state variables} in Problem~(P2) are $x$, $y$, $\theta$, $z_1$, $z_2$ and $z_3$, and the {\em control variable} is $u$.  There is a fatal flaw about Problem~(P2), however:  With the initial course $\theta(0)$ specified, there does not exist a solution, unless the choice of $\theta(0)$ happens to be an optimal solution of Problem~(P1).  We note that, in general, with no restrictions on the rate $\dot{\theta}$ of the heading angle in Problem~(P2), the control function $\theta$ will have an impulsive (Dirac delta) function component with a jump at $t = 0$, which is not allowed by the Maximum Principle we will consider.  

Moreover, in a realistic model, sharp turns are prohibited by the physical capabilities of the observer vehicle.  Therefore, a sensible thing to do is to impose a bound on the turning rate $\dot\theta(t) = u(t)/v$, with $u:[0,t_f]\to\dR$ the new variable.  For example, $|u(t)|\le v\,a$ is added to the model, where $a$ constitutes an upper bound for the {\em curvature} of the path and thus $v/a$ is the corresponding lower bound on the {\em turning radius}.  This restriction motivates us to pose the following realistic problem~(P3), instead of Problems~(P1) and (P2). 
\[
\mbox{(P3) }\left\{\begin{array}{rll}
\mbox{Solve} & \ \mbox{Problem~(P1)} & \hspace*{-20mm}\ \  \\[4mm] 
\mbox{subject to} 
 & \ \dot{\theta}(t) = u(t)/v\,, & \theta(0) = \theta_0\,, \\[2mm]
 & \ -v\,a \le u(t) \le v\,a\,, & \mbox{ a.e. } t\in[0,t_f]\,.
\end{array} \right.
\]
Note that when $\theta(0)$ is free and the control variable $u$ is not bounded (in that $a\to\infty$), Problem~(P3) reduces to Problem~(P1).  To the best knowledge of the author, Problem~(P3) has not previously been studied in the literature, analytically or numerically.

\subsection{Optimality Conditions for Problem~(P3)}

In this section, we will derive the necessary conditions of optimality for Problem~(P3), using the Maximum Principle.  Various forms of the Maximum Principle and their proofs can be found in a number of reference books -- see, for example, \cite[Theorem~1]{PonBolGanMis1986}, \cite[Chapter~7]{Hestenes1966}, \cite[Theorem~6.4.1]{Vinter2000}, \cite[Theorem~6.37]{Mordukhovich2006}, and \cite[Theorem~22.2]{Clarke2013}.  We will state the Maximum Principle suitably utilizing these references for our setting and notation.

First we define the state variable vector $s := (x,y,\theta,z_1,z_2,z_3)$ for brevity.  Note that $s(t)\in\dR^6$. Then we define the Hamiltonian function $H:\dR^6 \times \dR \times \dR^6 \to \dR$ for Problem~$(P3)$ in a standard manner as
\begin{equation} \label{Ham}
H(s,u,\lambda) := \lambda_x\,v\,\cos\theta + \lambda_y\,v\,\sin\theta + \lambda_\theta\,\frac{u}{v} +\lambda_{z_1}\,\frac{x^2}{(x^2 + y^2)^2} + \lambda_{z_2}\,\frac{y^2}{(x^2 + y^2)^2} + \lambda_{z_3}\,\frac{x\,y}{(x^2 + y^2)^2}\,,
\end{equation}
where $\lambda(t) := (\lambda_x(t),\lambda_y(t),\lambda_\theta(t),\lambda_{z_1}(t),\lambda_{z_2}(t),\lambda_{z_3}(t))\in\dR^6$ is the adjoint (or costate) variable vector.  In~\eqref{Ham}, we do not show the dependence of the variables on $t$, for clarity in appearance.  Keeping up with the tradition we introduce the notation
\[
H[t] := H(s(t),u(t),\lambda(t))\,.
\]
We also define the function $\varphi:\dR^6\to\dR$ such that
\[
\varphi(s(t_f)) := z_3^2(t_f) - z_1(t_f)\,z_2(t_f)\,.
\]
The adjoint variable vector is assumed to satisfy the differential equation and the boundary condition (see e.g.~\cite{Hestenes1966})
\begin{equation} \label{eq:adjoint}
\dot{\lambda}(t) := -H_s[t]\,,\quad \lambda(t_f) = \varphi_s (s(t_f)),
\end{equation}
where $H_s := \partial H / \partial s$ and $\varphi_s := \partial\varphi / \partial s$.  The boundary condition in~\eqref{eq:adjoint} is referred to as the {\em transversality condition}.  Componentwise, \eqref{eq:adjoint} can be re-written as
\begin{subequations}
\begin{eqnarray}
&& \dot{\lambda}_x(t) = -H_{x}[t]\,,\quad\ \  \lambda_x(t_f) = 0\,, \label{lambda_x}\\[1mm]
&& \dot{\lambda}_y(t) = -H_{y}[t]\,,\quad\ \  \lambda_y(t_f) = 0\,, \label{lambda_y}\\[1mm]
&& \dot{\lambda}_\theta(t) = -H_{\theta}[t]\,,\quad\ \  \lambda_\theta(t_f) = 0\,, \label{lambda_theta}\\[1mm]
&& \dot{\lambda}_{z_1}(t) = -H_{z_1}[t]\,,\quad  \lambda_{z_1}(t_f) = -z_2(t_f)\,, \label{lambda_z1}\\[1mm]
&& \dot{\lambda}_{z_2}(t) = -H_{z_2}[t]\,,\quad  \lambda_{z_2}(t_f) = -z_1(t_f)\,, \label{lambda_z2}\\[1mm]
&& \dot{\lambda}_{z_3} (t)= -H_{z_3}[t]\,,\quad  \lambda_{z_3}(t_f) = 2\,z_3(t_f)\,. \label{lambda_z3}
\end{eqnarray}
\end{subequations}
Clearly, $H_{z_i} = 0$, for $i = 1,2,3$.  Then \eqref{lambda_z1}--\eqref{lambda_z3} yield
\[
\lambda_{z_1}(t) = -z_2(t_f) =: -\zb_2\,,\qquad 
\lambda_{z_2}(t) = -z_1(t_f) =: -\zb_1\,,\qquad 
\lambda_{z_3}(t) = 2\,z_3(t_f) =: 2\,\zb_3\,,
\]
for all $t\in[0,t_f]$.  After expanding the right-hand sides of the ODEs in \eqref{lambda_x}--\eqref{lambda_theta}, and carrying out manipulations, one gets
\begin{subequations}
\begin{eqnarray}
&& \dot{\lambda}_x(t) = \frac{-2\zb_2 x^3(t) + 6\zb_3 x^2(t)y(t) + 2(\zb_2 - 2\zb_1) x(t)y^2(t) - 2\zb_3 y^3(t)}{(x^2(t) + y^2(t))^3}\,,\ \ \lambda_x(t_f) = 0\,, \label{ODE_lambda_x}\\[2mm]
&& \dot{\lambda}_y(t) = \frac{-2\zb_1 y^3(t) + 6\zb_3 y^2(t)x(t) + 2(\zb_1 - 2\zb_2) y(t)x^2(t) - 2\zb_3 x^3(t)}{(x^2(t) + y^2(t))^3}\,,\ \ \lambda_y(t_f) = 0\,, \label{ODE_lambda_y}\\[2mm]
&& \dot{\lambda}_\theta(t) = v \left(\lambda_x(t)\,\sin\theta(t) - \lambda_y(t)\,\cos\theta(t)\right),\quad\ \  \lambda_\theta(t_f) = 0\,. \label{ODE_lambda_theta}
\end{eqnarray}
\end{subequations}

\noindent
{\bf Maximum Principle} \\[1mm]
Suppose that the pair $(s,u)\in W^{1,2}([0,t_f];\dR^6) \times L^\infty([0,t_f];\dR)$ is optimal for Problem~(P3).  Then there exists a continuous adjoint variable vector $\lambda\in W^{1,2}([0,t_f];\dR^6)$ as defined in~\eqref{eq:adjoint}, such that $\lambda(t)\neq{\bf 0}$ for all $t\in[0,t_f]$, and that, for a.e. $t\in[0,t_f]$,
\begin{equation}\label{eqn:u_opt}
u(t) = \argmin_{|w| \le va} H(s(t),w,\lambda(t))\,.
\end{equation}
We note that the Maximum Principle originally involves the maximization of an objective functional and asserts Condition~\eqref{eqn:u_opt} as the maximization of the Hamiltonian.  Since we have written down Problem~(P3) as a minimization problem, Condition~\eqref{eqn:u_opt} is expressed as the minimization of the Hamiltonian.  Owing to minimizing, rather than maximizing, the Hamiltonian, the {\em Maximum Principle} is sometimes referred to as the {\em Minimum Principle} by some authors, see e.g. \cite{OsmMau2012}.  

Using \eqref{Ham}, and noting that $v>0$ is constant, the condition~\eqref{eqn:u_opt} can simply be re-written as
\begin{equation}\label{eqn:u_opt1}
u(t) = \argmin_{|w| \le va} \lambda_\theta(t)\,w\,.
\end{equation}

\begin{proposition}  \label{prop:opt_contr}
The optimal control for Problem~(P3) is given by
\begin{equation}  \label{eqn:u_opt2}
u(t) = \left\{\begin{array}{rl}
v\,a\,, &\ \ \mbox{if\ \ } \lambda_\theta(t) < 0\,, \\[2mm]
-v\,a\,, &\ \ \mbox{if\ \ } \lambda_\theta(t) > 0\,, \\[2mm]
\mbox{undetermined}\,, &\ \ \mbox{if\ \ } \lambda_\theta(t) = 0\,,
\end{array} \right.
\end{equation}
for a.e. $t\in[0,t_f]$.
\end{proposition}
\begin{proof}
The proof immediately follows from \eqref{eqn:u_opt1}.
\end{proof}

\noindent \\[-4mm]
{\bf Bang--bang and singular control} \\[1mm]
If $\lambda_\theta \neq 0$ for a.e. $t\in[r',r'']\subset[0,t_f]$, i.e., $\lambda_\theta(t) = 0$ only at isolated points in $[r',r'']$, then the optimal control is referred to as {\em bang--bang} in the interval $[r',r'']$.  In this case, the optimal control might {\em switch} from $\lim_{t\to t_1^-} u(t) = -va$\ \ to\ \ $\lim_{t\to t_1^+} u(t) = va$, or from $\lim_{t\to t_1^-} u(t) = va$\ \ to\ \ $\lim_{t\to t_1^+} u(t) = -va$, at a {\em switching time}  $t_1\in [r',r'']$, with $\lambda_\theta(t_1) = 0$.

If $\lambda_\theta(t) = 0$ for a.e. $t\in[t',t'']\subset[0,t_f]$, then the optimal control is said to be {\em singular} in the interval $t\in[t',t'']$.  Note that the optimal control might also switch from a bang arc to a singular arc, and vice versa.

Since a switching will occur when there is typically a change in the sign of $\lambda_\theta(t)$, the continuous function $\lambda_\theta:[0,t_f]\to\dR$ appearing in the conditions in \eqref{eqn:u_opt2} is suitably named as the {\em switching function}.

Next we state a simple but useful fact about the switching function $\lambda_\theta$.
\begin{fact}  \label{deriv_switching}
$\dot{\lambda}_\theta(t_f) = 0$.
\end{fact}
\begin{proof}
We have that $\dot{\lambda}_\theta(t_f) = v \left(\lambda_x(t_f)\,\sin\theta(t_f) - \lambda_y(t_f)\,\cos\theta(t_f)\right)$ from~\eqref{ODE_lambda_theta}, and that $\lambda_x(t_f) = \lambda_y(t_f) = 0$ from \eqref{ODE_lambda_x} and \eqref{ODE_lambda_y}, respectively, with direct substitution yielding the required result.
\end{proof}

\begin{proposition}  \label{prop:singular}
If the optimal control is singular, then
\begin{equation}  \label{eqn:theta_sing}
\theta_{\mbox{\footnotesize\rm sing}}(t) = \arctan\frac{\lambda_y(t)}{\lambda_x(t)}\,.
\end{equation}
\end{proposition}
\begin{proof}
Suppose that $u(t)$ is singular for $t\in[t',t'']$.  Then $\lambda_\theta(t) = 0$, and $\dot{\lambda}_\theta(t) = 0$, on $t\in[t',t'']$, and so from~\eqref{ODE_lambda_theta}, also using $v\neq0$,
\[
\lambda_x(t)\,\sin\theta_{\sing}(t) - \lambda_y(t)\,\cos\theta_{\sing}(t) = 0\,.
\]
Solving this equation for $\theta_{\sing}(t)$ yields \eqref{eqn:theta_sing}, as required.
\end{proof}

\begin{remark}  \rm
As we stressed before, a solution to Problem~(P3) can only be found approximately by means of numerical methods.  However, the analytical information obtained so far on the optimality of a solution, in particular Propositions~\ref{prop:opt_contr} and \ref{prop:singular}, and Fact~\ref{deriv_switching} along with~\eqref{ODE_lambda_theta}, can suitably be used in checking whether or not an approximate solution obtained by a numerical method will (approximately, in other words, numerically) verify the optimality conditions.
\proofbox
\end{remark}

\subsection{Optimality Conditions for Problem~(P1)}

We have obtained above the necessary optimality conditions for Problem~(P3) first, since they are more general than the conditions we will obtain for Problem~(P1).  In other words, the optimality conditions for Problem~(P1) are in some sense a special case of those for Problem~(P3).  First, we define the state variable vector for Problem~(P1) as $\st := (x,y,z_1,z_2,z_3)$, for which the control variable is nothing but $\theta$.  We note that $s(t)\in\dR^5$. Then the Hamiltonian function $\Ht:\dR^5 \times \dR \times \dR^5 \to \dR$ for Problem~$(P1)$ is defined as
\begin{equation} \label{Hamtilde}
\Ht(\st,\theta,\lambdat) := \lambda_x\,v\,\cos\theta + \lambda_y\,v\,\sin\theta + \lambda_{z_1}\,\frac{x^2}{(x^2 + y^2)^2} + \lambda_{z_2}\,\frac{y^2}{(x^2 + y^2)^2} + \lambda_{z_3}\,\frac{x\,y}{(x^2 + y^2)^2}\,,
\end{equation}
with $\lambdat(t) := (\lambda_x(t),\lambda_y(t),\lambda_{z_1}(t),\lambda_{z_2}(t),\lambda_{z_3}(t))\in\dR^5$ as the adjoint variable vector. Similarly we define
\[
\Ht[t] := \Ht(\st(t),\theta(t),\lambdat(t))\,,
\]
and $\varphit:\dR^5\to\dR$ such that $\varphit(s(t_f)) := z_3^2(t_f) - z_1(t_f)\,z_2(t_f)$\,.  The adjoint variables in this case satisfy, similarly, the following ODEs.
\begin{subequations}
\begin{eqnarray}
&& \dot{\lambda}_x(t) = -\Ht_{x}[t] = -H_{x}[t]\,,\quad\ \  \lambda_x(t_f) = 0\,, \label{lambdat_x}\\[1mm]
&& \dot{\lambda}_y(t) = -\Ht_{y}[t] = -H_{y}[t]\,,\quad\ \  \lambda_y(t_f) = 0\,, \label{lambdat_y}\\[1mm]
&& \dot{\lambda}_{z_1}(t) = -\Ht_{z_1}[t] = -H_{z_1}[t]\,,\quad  \lambda_{z_1}(t_f) = -z_2(t_f)\,, \label{lambdat_z1}\\[1mm]
&& \dot{\lambda}_{z_2}(t) = -\Ht_{z_2}[t] = -H_{z_2}[t]\,,\quad  \lambda_{z_2}(t_f) = -z_1(t_f)\,, \label{lambdat_z2}\\[1mm]
&& \dot{\lambda}_{z_3} (t) = -\Ht_{z_3}[t] = -H_{z_3}[t]\,,\quad  \lambda_{z_3}(t_f) = 2\,z_3(t_f)\,. \label{lambdat_z3}
\end{eqnarray}
\end{subequations}
Again, similarly, since $H_{z_i} = 0$, for $i = 1,2,3$, one gets
$\lambda_{z_1}(t) = -z_2(t_f) =: -\zb_2$, $\lambda_{z_2}(t) = -z_1(t_f) =: -\zb_1$ and $\lambda_{z_3}(t) = 2\,z_3(t_f) =: 2\,\zb_3$, for all $t\in[0,t_f]$.

The Maximum Principle stated in the previous section applies to this case as follows: Suppose that the pair $(\st,\theta)\in W^{1,2}([0,t_f];\dR^5) \times L^2([0,t_f];\dR)$ is optimal for Problem~(P1).  Then there exists a continuous adjoint variable vector $\lambdat\in W^{1,2}([0,t_f];\dR^5)$ as defined in~\eqref{lambdat_x}--\eqref{lambdat_z3}, such that $\lambdat(t)\neq{\bf 0}$ for all $t\in[0,t_f]$, and that, for a.e.~$t\in[0,t_f]$,
\begin{equation}\label{eqn:theta_opt}
\theta(t) = \argmin_{\alpha} \Ht(\st(t),\alpha,\lambdat(t))\,,
\end{equation}
i.e.,
\begin{equation}\label{eqn:theta_opt1}
\Ht_\theta[t] := \frac{\partial\Ht}{\partial\theta}(\st(t),\theta(t),\lambdat(t)) = 0\,.
\end{equation}

\begin{proposition}  \label{prop:P1_opt_contr}
The optimal control for Problem~(P1) is given by
\begin{equation}  \label{eqn:theta_opt2}
\theta(t) = \arctan\frac{\lambda_y(t)}{\lambda_x(t)}\,,
\end{equation}
for a.e.~$t\in[0,t_f]$.
\end{proposition}
\begin{proof}
The proof follows from \eqref{eqn:theta_opt1} first by writing out the equation
\[
\Ht_\theta[t] = v \left(\lambda_x(t)\,\sin\theta(t) - \lambda_y(t)\,\cos\theta(t)\right) = 0\,,
\]
and then, with $v\neq 0$, by solving the equation for $\theta(t)$.
\end{proof}

\begin{remark}  \rm
We note from~\eqref{eqn:theta_opt2} and the transversality conditions $\lambda_x(t_f) = \lambda_y(t_f) = 0$ in \eqref{lambdat_x}--\eqref{lambdat_y} that the optimal control at the terminal point, $\theta(t_f)$, is indeterminate.
\proofbox
\end{remark}

\begin{corollary}  \label{cor:optcontr}
The optimal control u(t) for Problem~(P3) is a concatenation of bang arcs, with $\dot{\theta}(t) = a$ or $-a$, and {\em interior arcs}, with $\theta(t) = \arctan(\lambda_y(t) / \lambda_x(t))$, where $\lambda_x$ and $\lambda_y$ solve \eqref{ODE_lambda_x}--\eqref{ODE_lambda_y}.
\end{corollary}
\begin{proof}
We note that the ODEs \eqref{lambdat_x}--\eqref{lambdat_y} that $\lambda_x(t)$ and $\lambda_y(t)$ satisfy are the same as those in \eqref{ODE_lambda_x}--\eqref{ODE_lambda_y}.  Therefore, we conclude that the expression \eqref{eqn:theta_opt2} for $\theta(t)$ given in~Proposition~\ref{prop:P1_opt_contr} is the same as the expression \eqref{eqn:theta_sing} for $\theta_{\sing}(t)$ given in~Proposition~\ref{prop:singular} for the case of singular control for Problem~(P3).  Then combining this observation with the expression~\eqref{eqn:u_opt2} for $u_{\sing}(t)$ given in~Proposition~\ref{prop:opt_contr} furnishes the proof.
\end{proof}

\section{Numerical Experiments}  
\label{numerical}

For the numerical experiments, we consider the target localization example problem studied in~\cite{OshDav1999}:  An observer travelling at a constant speed of $v = 40$ m/s (or 144 km/h) is 5 km away from a stationary target.  For convenience the target's Cartesian coordinates are chosen to be the origin $(0,0)$ and those of the observer to be $(5,0)$.  Then the relative observer coordinates are nothing but $(x_0,y_0) = (5,0)$.

In what follows, we find optimal observer trajectories by solving Problems~(P1) and (P3).  We obtain numerical, i.e., approximate, solutions via Euler discretization of both problems, utilizing 1,000 grid points.  To solve the discretized problem (which is a finite-dimensional optimization problem with around 7,000 variables) we employ the optimization modelling language AMPL~\cite{AMPL}, which is paired up with the optimization software Ipopt~\cite{WacBie2006} that uses an interior point method.  Theory and applications of discretized solutions of constrained optimal control problems can be found in \cite{BanKay2013a, DonHagMal2000, Hager2000, KayMar2007, KayMau2014}.

\subsection{Numerical Solution of Problem~(P1)}
\label{sec:P1}  

Figure~\ref{fig:P1a} depicts discretized solutions of Problem~(P1) for various scenarios involving different lengths of the time interval $[0, t_f]$, with $t_f = 12.5, 25, 50, 75, 100, 112.5, 118.75$ s.  These times may be interpreted as the times allowed to make the bearing angle measurements, after which they can be used in estimating the target position.  In~\cite{OshDav1999, HamLiuHilGon1989}, numerical solutions of such problems are presented in terms of a nondimensional constant, $K = v\,t_f / r_0$, where $r_0$ is the initial range.  For the example we are studying, $r_0 = 5$.  In other words, the paths we generate are for $K = 0.1, 0.2, 0.4, 0.6, 0.8, 0.9, 0.95$, which are the same values as those used in~\cite{OshDav1999}.  As far as one can judge from the appearance, the observer paths shown in Figure~\ref{fig:P1a} are the same as those displayed in Figure~4 in~\cite{OshDav1999}.  

\begin{figure}[t!]
\begin{subfigure}{\textwidth}
\centering
\includegraphics[width=120mm]{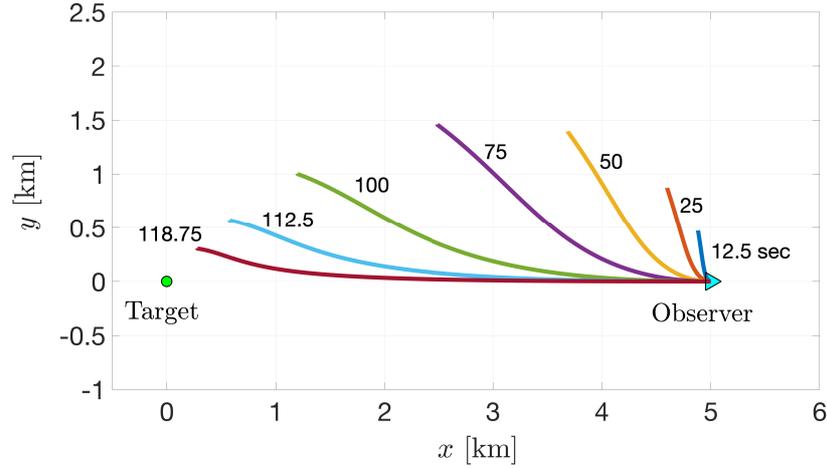}
\caption{\sf Optimal observer paths.}
\label{fig:P1a}
\end{subfigure}
\begin{subfigure}{\textwidth}
\centering
\includegraphics[width=120mm]{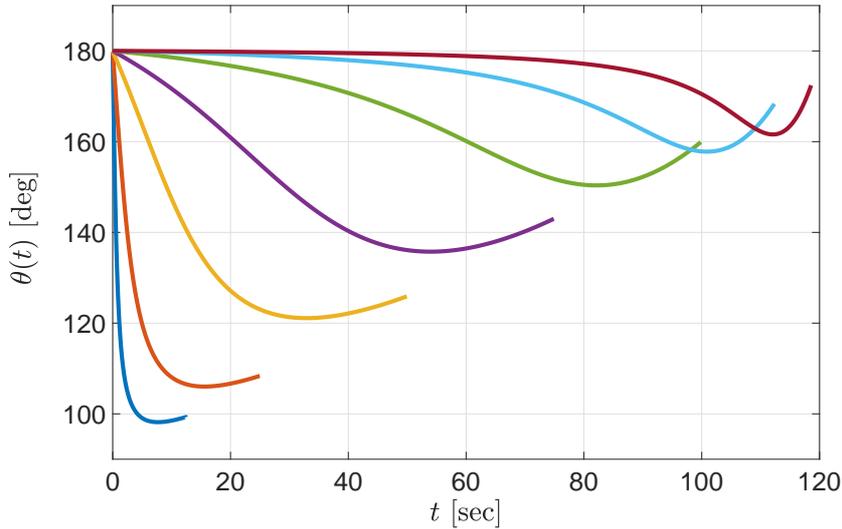}
\caption{\sf Optimal course histories.}
\label{fig:P1b}
\end{subfigure}
\caption{\sf Problem~(P1)---Optimal observer paths and course histories.}
\label{fig:P1}
\end{figure}

We have checked that these solutions indeed verify the necessary conditions of optimality for the continuous-time problem~(P1), in particular the expression given for $\theta(t)$ in Proposition~\ref{prop:P1_opt_contr}.  These conditions are not verified in~\cite{OshDav1999}.  Here, we have added the graph of $\theta(t)$ in Figure~\ref{fig:P1b}, which is otherwise not provided in~\cite{OshDav1999}.   Figure~\ref{fig:P1b} motivates/justifies an extension of Problem~(P1), by adding the initial course as well as a bound on the curvature, namely Problem~(P3), as elaborated further below.

It can be clearly seen from Figure~\ref{fig:P1b} that the optimal initial course $\theta(0)$ turns out to be 180 deg for all the different time intervals $[0, t_f]$ that was considered.  Solving these problems with a ``free'' initial course should perhaps be viewed as unnatural, because an observer would typically be mounted on a vessel/vehicle that cannot instantaneously ``turn around'' to position its heading to be at 180 deg.  In the example we are studying, the vehicle might as well be travelling east (at $40$ km/s!) rather than west, in which case $\theta(0) = 0$ deg instead of $\theta(0) = 180$ deg.  In practice, it will be necessary to prescribe $\theta(0)$ as the currently given course of the vehicle at $t = 0$, and one will have to find the evolution of the course for the horizon~$[0, t_f]$, which is required to attain maximum information, with this initial course.

Problem~(P1) does not impose a lower bound on the instantaneous turning radius $v/|\dot{\theta}(t)|$ of the observer, in other words, an upper bound on the rate of change of the course, or the instantaneous curvature, $|\dot{\theta}(t)|$.  Figure~\ref{fig:P1b} reveals that, especially when $t_f$ is small, the change in $\theta$ is abrupt, resulting in sharp turns.  In such cases, (i) the vehicle where the observer is mounted on may not be physically capable of achieving these sharp turns and (ii) the observer or sonar measurements may get corrupted because of these highly accelerated manoeuvres.  So it is natural to impose a bound on the turning rate.

Recall that both of the above concerns, namely the prescription of the initial course as well as imposition of an upper bound on the curvature, were included into the formulation of Problem~(P3), the solution of which for the current example, based on the example in~\cite{OshDav1999}, follows next.

\subsection{Numerical Solution of Problem~(P3)}

In Problem~(P3), we have prescribed the initial course $\theta(0) = 0$ deg, and imposed the upper bound for the curvature to be $5$ deg/s, for the example in~\cite{OshDav1999}. Optimal observer paths and course histories obtained as discretized solutions of Problem~(P3) are depicted in Figure~\ref{fig:P3}, for the terminal times, $t_f = 50, 80, 100, 120, 140, 150, 155, 160$ s.

\begin{figure}[t!]
\begin{subfigure}{\textwidth}
\centering
\includegraphics[width=120mm]{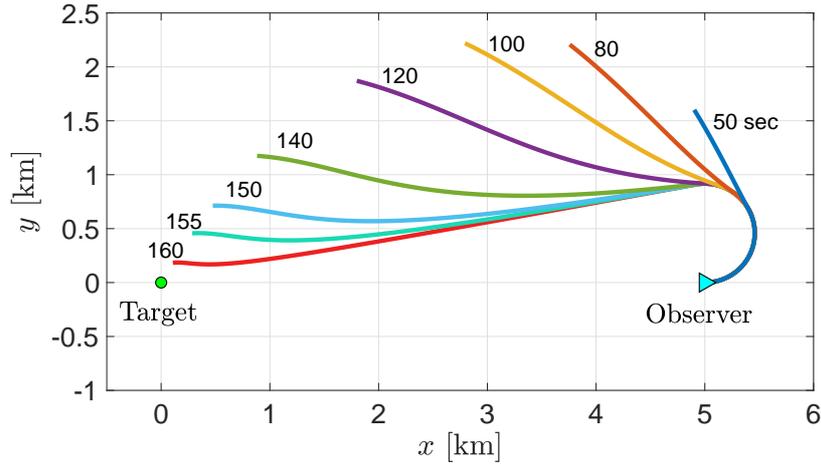}
\caption{\sf Optimal observer paths.}
\label{fig:P3a}
\end{subfigure}
\begin{subfigure}{\textwidth}
\centering
\includegraphics[width=120mm]{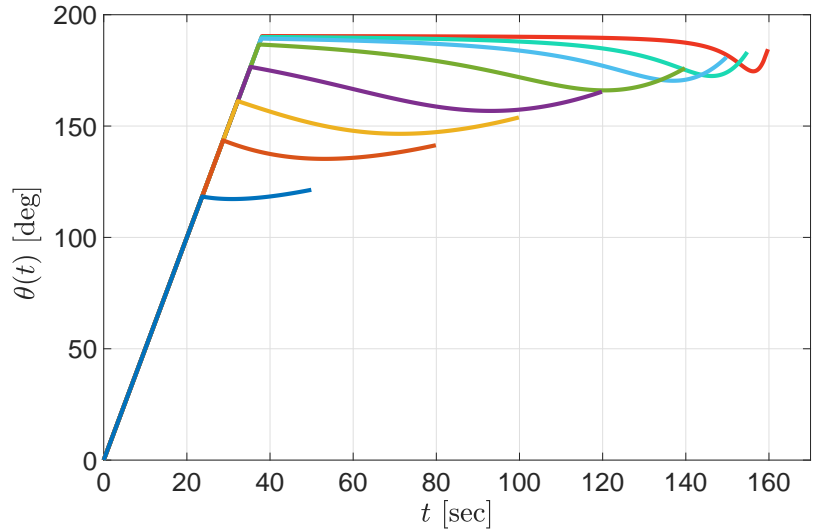}
\caption{\sf Optimal course histories.}
\label{fig:P3b}
\end{subfigure}
\caption{\sf Problem~(P3)---Optimal observer paths and course histories.}
\label{fig:P3}
\end{figure}

As can be seen in Figure~\ref{fig:P3a}, the initial segment of the path looks to be circular, which is verified by the linearity of $\theta$ in Figure~\ref{fig:P3b}, initially.  This is further reconfirmed from the graph of $\dot{\theta}(t)$ in Figure~\ref{fig:P3control}:  The initial segments of the trajectories corresponding to each $t_f$ is of constant curvature of 5 deg/s.  This equivalently corresponds to a constant turning radius of about 460 m for the observer vehicle.

\begin{figure}[t!]
\begin{subfigure}{\textwidth}
\centering
\includegraphics[width=120mm]{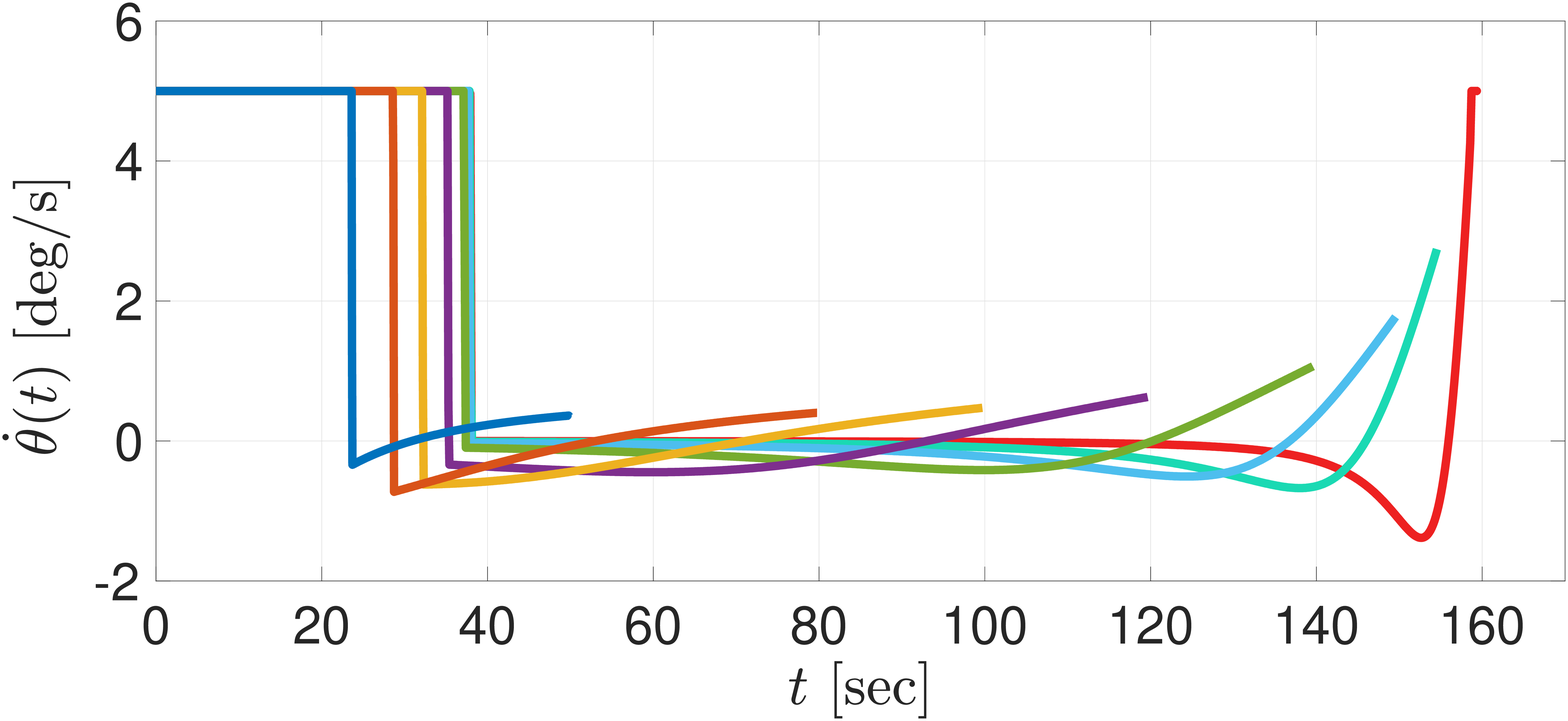}
\caption{\sf Optimal course speed $\dot{\theta}(t)$.}
\label{fig:P3controla}
\end{subfigure}
\begin{subfigure}{\textwidth}
\centering
\includegraphics[width=120mm]{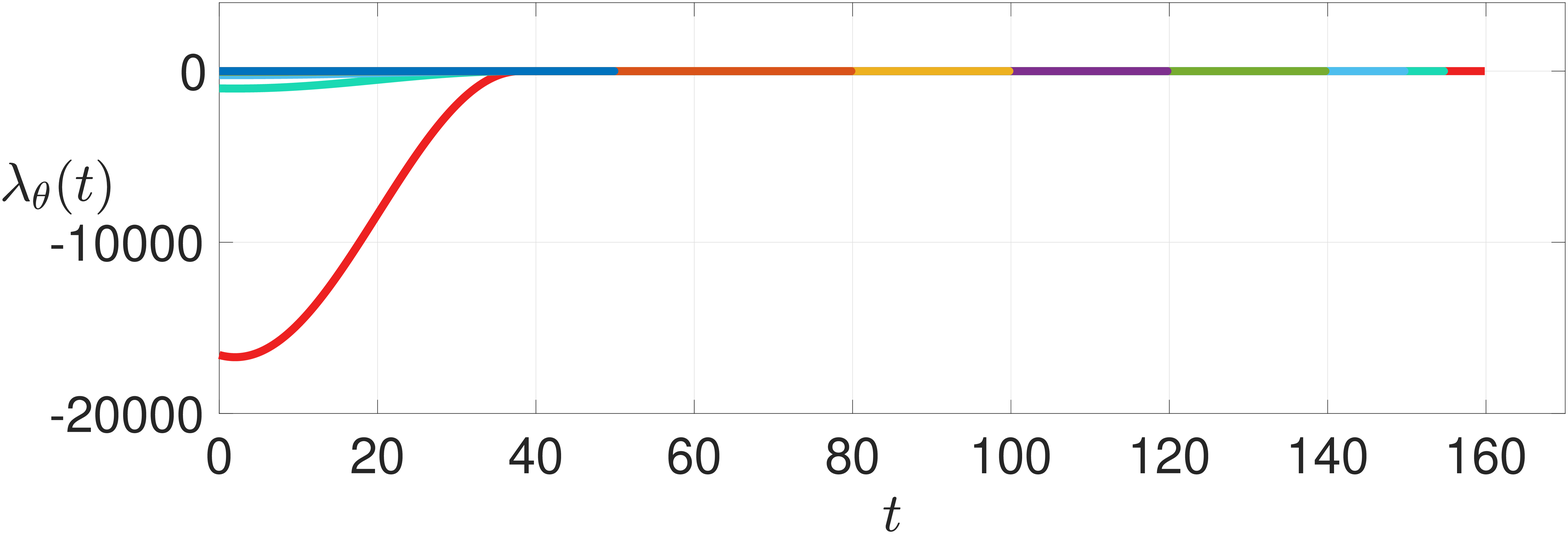}
\includegraphics[width=120mm]{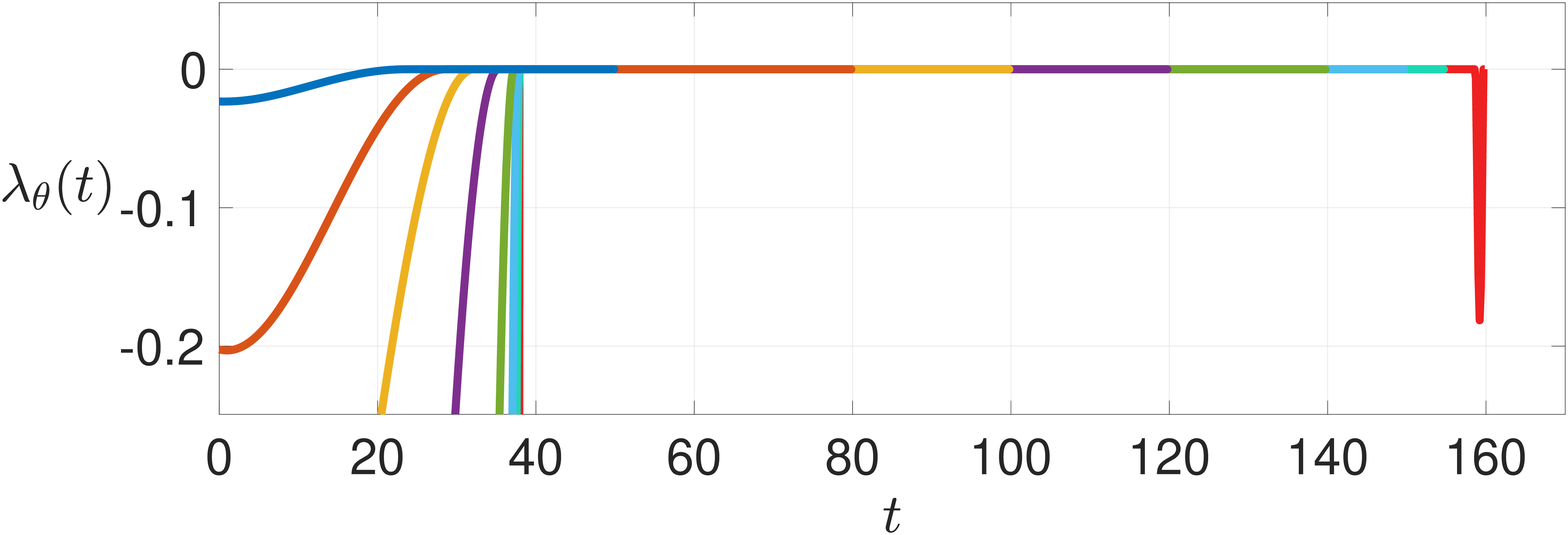}
\caption{\sf Switching function $\lambda_\theta(t)$.}
\label{fig:P3controlb}
\end{subfigure}
\caption{\sf Problem~(P3)---Optimal course speeds and switching functions.}
\label{fig:P3control}
\end{figure}

Figure~\ref{fig:P3control} conveys further information:  The optimal control switches from the bang arc $\dot{\theta}(t) = 5$ deg/s to singular control, which is an interior arc, i.e., the constraint on the curvature is inactive, as also described in Corollary~\ref{cor:optcontr}.  The numerical experiments conducted so far suggest that the observer trajectory reaches the target position with infinite information when $t_f \approx 163.3$.  

Broadly speaking, by observing the graphs in~Figure~\ref{fig:P3controla}, while the optimal control structure is of {\em bang--singular} type for $t_f < 160$, the structure of the optimal control for $t_f \ge 160$ appears to be of {\em bang--singular--bang} type.  A verification of this can be carried out by using the scaled graphs of the switching function $\lambda_\theta(t)$ given in Figure~\ref{fig:P3controlb}.

In Figure~\ref{fig:P3controlb}, we show two sets of graphs of the switching function $\lambda_\theta(t)$ corresponding to each of the eight values of $t_f$.  The vertical axis of the first set is scaled so that the switching function for the case when $t = 160$ is visible as a whole.  However, the vertical axis of the second graph is scaled in such a way that the switching times for the cases of each value of $t_f$ can be viewed relatively easily.  

Indeed, from the second set of graphs, along with the graphs in~Figure~\ref{fig:P3control}, we obtain a numerical/graphical verification of the optimal control law asserted in Proposition~\ref{prop:opt_contr}, by comparing the signs of the switching function in Figure~\ref{fig:P3controlb} and the graphs of the course speed in Figure~\ref{fig:P3controla}.  Over the initial segment where the switching times occur around 20 to 40 seconds depending on the value of $t_f$, it is clearly seen that $\dot{\theta}(t) = 5$ since $\lambda_\theta(t) < 0$, by the Maximum Principle.  We note that the singular optimal control is characterized by the segment where $\lambda_\theta(t) \equiv 0$.  See also Proposition~\ref{prop:singular} for an expression of the singular optimal control.

It is particularly interesting to observe that while the overall scale of the first set of graphs of $\lambda_\theta(t)$ is $2\times10^5$, by looking at the same graphs over a scale of $2\times10^{-1}$, it is possible to capture that $\lambda_\theta(t) < 0$ over the very short end-piece of the time horizon $[0,160]$.  This observation verifies that, by the Maximum Principle, $\dot{\theta}(t) = 5$ during the (short) end-piece of the time horizon $[0,160]$.  Thus, overall, we are able to reconfirm the {\em bang--singular--bang} type structure for the optimal control when $t_f = 160$.

Note that the Maximum Principle provides the necessary conditions of optimality.  So, what one can vouch for is the stationarity of a solution.  However, we describe the numerical solutions that we obtain here as {\em optimal} in this loose sense.

It turns out, via our numerical experiments, that there exist other solutions which also satisfy the Maximum Principle.  Of course, there are, first of all, symmetric solutions, namely the paths symmetric about the $x$-axis in Figure~\ref{fig:P3a}, which are somewhat trivial to figure.  On the other hand, Figure~\ref{fig:P3loca} depicts solutions which are not so trivial (in that they are not symmetric about the $x$-axis) and are at best only locally optimal.

\begin{figure}[t!]
\begin{subfigure}{\textwidth}
\centering
\includegraphics[width=120mm]{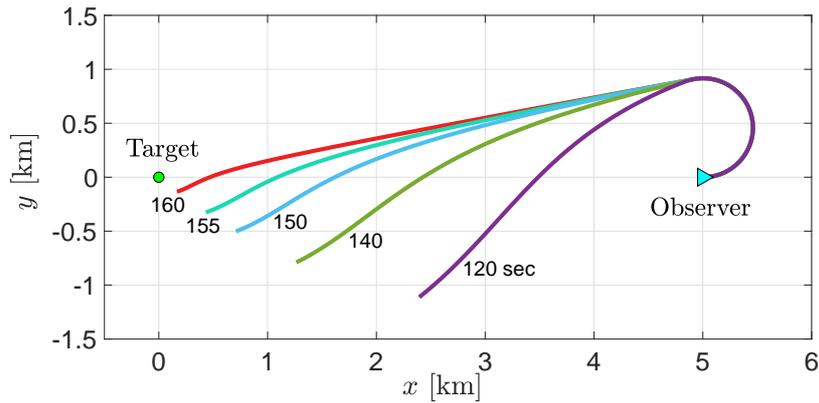}
\caption{\sf Locally-optimal observer paths.}
\label{fig:P3loca}
\end{subfigure}
\begin{subfigure}{\textwidth}
\centering
\includegraphics[width=120mm]{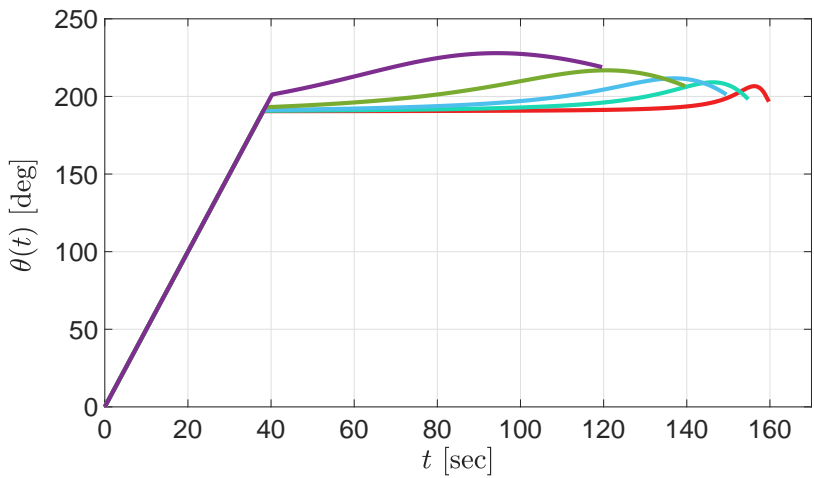}
\caption{\sf Locally-optimal course histories.}
\label{fig:P3locb}
\end{subfigure}
\caption{\sf Problem~(P3)---Locally-optimal observer paths and course histories.}
\label{fig:P3loc}
\end{figure}

In Table~\ref{table:objval}, we list the values of the ``normalized'' determinant of the Fisher information matrix, $\sigma^4\det(F(\theta))$, where $\sigma^2$ is the constant variance of the bearing angle measurement process.  For each of the solutions depicted in Figures~\ref{fig:P3loca} and \ref{fig:P3a}, the tabulated objective functional values are correct to four significant figures.  This level of accuracy in the optimal values was achieved by discretizing Problem~(P3) via the trapezoidal rule, which is a second-order approximation scheme as opposed to the first-order Euler discretization scheme.

\begin{table}[h]
\centering
{\small
\begin{tabular}{r | l l} 
  & \multicolumn{2}{c}{$\sigma^4\det(F(\theta))$} \\[1mm] \cline{2-3} \\ \\[-7mm]
 $t_f$ & Fig.~\ref{fig:P3a}  & Fig.~\ref{fig:P3loca} \\[1mm]
 \hline\hline \\ \\[-8mm]
 50 & \ \ \ \ \,0.03058 &  \\
 80 & \ \ \ \ \,0.2637 &  \\
100 & \ \ \ \ \,0.8822 &  \\
120 & \ \ \ \ \,3.149 & \ \ \ \ \,2.285 \\
 140 & \ \ \ 16.82 & \ \ \ 14.69 \\
 150 & \ \ \ 63.66 & \ \ \ 59.48 \\
 155 & \ \,183.2 & \ \,176.1 \\
 160 & 1376 & 1357 \\
 \end{tabular}
\caption{\sf Problem~(P3)---Locally-optimal objective functional values.}
\label{table:objval}}
\end{table}

The numbers listed in Table~\ref{table:objval} imply that the solutions presented in Figure~\ref{fig:P3a} are certainly better solutions. However, we would still like to note that the solution curves depicted in Figure~\ref{fig:P3loca} satisfy the Maximum Principle:  We observe, for the values of $t_f$ that was used, that the course speed $\dot{\theta}(t) = 5$ in Figure~\ref{fig:P3loccontrola}, while $\lambda_\theta(t) < 0$ in Figure~\ref{fig:P3loccontrolb}, and that the control is singular when $\lambda_\theta(t) \equiv 0$ over a time interval.  We also note that when $t_f = 160$ s the optimal control is of {\em bang--singular--bang} type, which is reconfirmed by the signs of the switching function $\lambda_\theta(t)$ in the second graph of Figure~\ref{fig:P3loccontrolb}.

\begin{figure}[t!]
\begin{subfigure}{\textwidth}
\centering
\includegraphics[width=120mm]{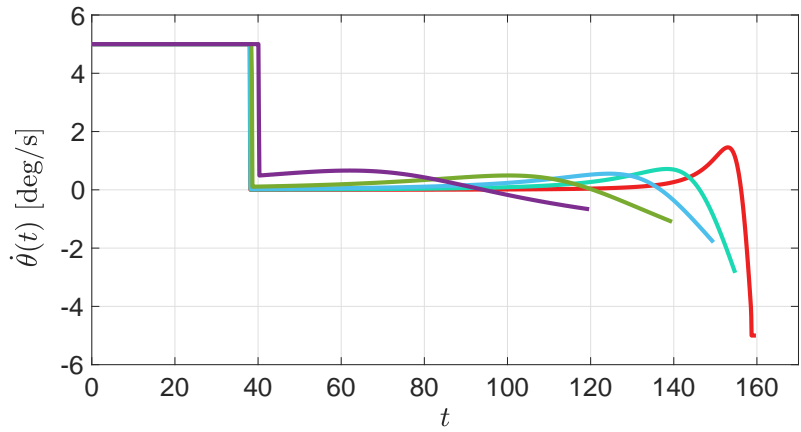}
\caption{\sf Locally-optimal course speed $\dot{\theta}(t)$.}
\label{fig:P3loccontrola}
\end{subfigure}
\begin{subfigure}{\textwidth}
\centering
\includegraphics[width=120mm]{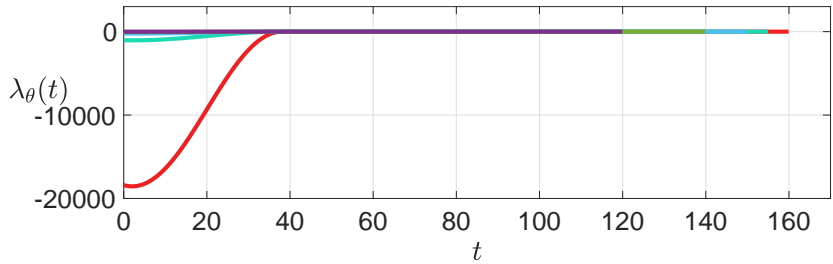}
\includegraphics[width=120mm]{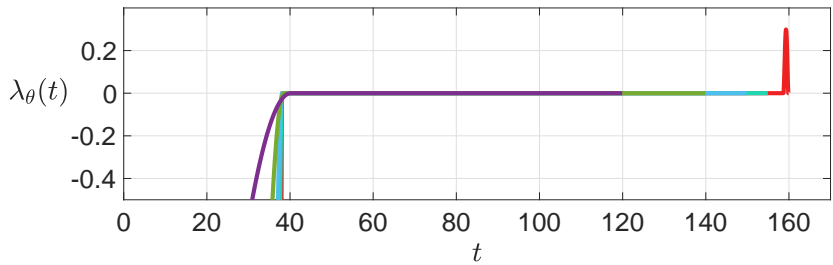}
\caption{\sf Switching function $\lambda_\theta(t)$.}
\label{fig:P3loccontrolb}
\end{subfigure}
\caption{\sf Problem~(P3)---Locally-optimal course speeds and switching functions.}
\label{fig:P3loccontrol}
\end{figure}

In all numerical solutions of Problem~(P3), we have also checked and reconfirmed that $\dot{\lambda}(t_f) = 0$ from Fact~\ref{deriv_switching} is indeed verified.

One final remark regarding computations is in order:  Discretized solutions over a coarse time grid can quickly provide the structure of a bang--singular or bang--singular--bang solution.  Once a coarse-grid solution revealing the structure is found, a more refined/accurate solution can be obtained by employing shooting-like methods as those provided in~\cite{KayMau2014, KayNoa2003, MauBueKimKay2005}.

\section{Conclusion and Discussion}
\label{conclusion}

We have studied the problem of maximizing information for an observer vehicle modelled as a point mass travelling at a constant speed and taking measurements of the bearing angle between the stationary target and the observer.  The optimal control model we have developed for this problem is different from those existing in the literature in that it prescribes the initial course, or heading angle, of the vehicle and it imposes a bound on the curvature, or a bound on the turning radius.

We derived the necessary conditions of optimality for the new problem via the Maximum Principle.  We obtained numerical solutions for an example configuration via Euler discretization and found that the optimal control for this example is either of {\em bang--singular} type or {\em bang--singular--bang} type.  We found out that there exist (nontrivially) multiple locally-optimal control solutions.  We verified numerically, or graphically, that the approximate solutions we have obtained satisfy the necessary conditions of optimality.

The following further extensions of the problem in this paper would be meaningful to make and study.

\noindent
{\bf Other measures of information.}  It would be interesting to maximize the trace of the Fisher information matrix, as another measure of information, instead of the determinant, which should be expected to lead to different optimal path solutions.

\noindent
{\bf Multiple observers.}  One would surely expect the information to increase if more than just one observer vehicle is employed.  Problem~(P3) can easily be extended to the case when there are multiple observer vehicles moving independently, subject to similar initial course and control constraints set for each vehicle.

\noindent
{\bf Moving target.}  The setting can be generalized to the case when the target is moving, in which the problem is referred to as {\em bearings-only tracking} rather than a {\em bearings-only localization}.  For example~\cite{PasCap1998} also considers the case when the target is moving along a straight line at constant speed, but of course without any of the other extensions we studied in the current paper.  To cater for a moving target, the objective functional of Problem~(P3) will need to be modified, as outlined in~\cite{PasCap1998}.  The new formulation might also be able to cater for more complicated manoeuvres of the target.

\noindent
{\bf State constraints.}  In many realistic situations, the observer may have to travel through the water surface, underwater, sky or terrain, which contains ``no-go'' areas modelled as constraints involving the states $x$ and $y$.  For example, the observer may have to stay a certain distance away from the target, as it happens in an example in~\cite{OshDav1999}.  In the example in~\cite{OshDav1999}, the only state constraint is given as $x^2(t) + y^2(t) \ge 4^2$, representing having to stay away from the threat from the target.  One may as well specify circular threat zones centred at various other points in a similar fashion. \\[1mm]
Problem~(P3) with state constraints is much more challenging both analytically and numerically, resulting in {\em boundary arcs} when a state constraint becomes active~\cite{HarSetVic1995}.  This extension would typically result in {\em bang--singular--boundary} arcs in a similar fashion to the type of solutions obtained for another example (a container crane) in~\cite{BanKay2013a}.

\noindent
{\bf Multiobjective optimization.}  Many optimization problems involve minimization or maximization of more than one objective which are conflicting in nature.  For instance, in addition to maximizing information in Problem~(P3), one may consider minimization of the terminal time $t_f$, ``simultaneously.''  This gives rise to a multiobjective optimal control problem, for which specialized theory and numerical methods are needed \cite{BonKay2010, KayMau2014}. It is interesting to note that the problem of minimizing the terminal time $t_f$, instead of maximizing information, with unit speed, i.e., $v = 1$, constitutes a special case of the Markov--Dubins problem \cite{Kaya2017, Kaya2019}.

\noindent
{\bf Formulation in polar coordinates.}  In some situations it might be more convenient to deal with the problem in polar coordinates, rather than Cartesian coordinates, although this may not constitute an extension on its own.  The first
two differential equations in Problem~(P3), which are in the Cartesian coordinates $x$ and $y$, can be replaced by those in the polar coordinates $r := (x^2 + y^2)^{1/2}$, and $\beta$, as defined in~\eqref{beta}:
\begin{eqnarray*}
&& \dot{r}(t) = v\,\cos(\beta(t) - \theta(t))\,,\qquad r(0) = r_0\,,
  \\[1mm]
&& \dot{\beta}(t) = \frac{v}{r(t)}\,\sin(\beta(t) - \theta(t))\,,\ \
   \beta(0) = \beta_0\,,
\end{eqnarray*}
where $r_0 := (x_0^2 + y_0^2)^{1/2}$ and $\beta_0 := \arctan(y_0/x_0)$, and \eqref{eqnx}--\eqref{eqny} have been used.  The Fisher information matrix, and the RHS's of the ensuing differential equations for $z_i$, $i=1,2,3$, in Problem~(P3) can also be simply written in terms of $r$ and $\beta$, by using $\sin\beta = y/r$ and $\cos\beta = x/r$.

\section*{Acknowledgments}
The author would like to offer his warm thanks to the anonymous reviewer for their careful reading of the manuscript and suggestions.  He acknowledges useful discussions with Sanjeev Arulampalam of Defence Science and Technology in 2016, after which he started to look at the problem studied in the current paper.

\end{document}